\documentclass[reqno]{amsart}
\usepackage{latexsym}
\usepackage{amssymb,amsthm,amsmath}

\theoremstyle{plain}
\newtheorem{theorem}{Theorem}
\newtheorem{lemma}[theorem]{Lemma}
\newtheorem{corollary}[theorem]{Corollary}
\newtheorem{proposition}[theorem]{Proposition}

\theoremstyle{definition}
\newtheorem{definition}[theorem]{Definition}
\theoremstyle{remark}
\newtheorem{remark}[theorem]{Remark}

\def\d#1{{#1\kern-0.4em\char"16\kern-0.1em}}
\def\D#1{{\raise0.2ex\hbox{-}\kern-0.4em #1}}
\reversemarginpar

\newcounter{zd}

\newcounter{zdr}[subsection]

\newcommand{\eps}{\varepsilon}

\def\F{{\cal F}}
\def\pa{\partial}

\def\cal{\mathcal}
\let\mib=\boldsymbol
\def\ae#1{\;(\hbox{\rm a.e. } #1)}

\def\R{{\bf R}}

\def\N{{\bf N}}

\def\Cbc#1{{{\rm C}^{\infty}_{c}(#1)}}

\def\pC#1#2{{{\rm C}^{#1}(#2)}}

\def\Dup#1#2{\langle#1,#2\rangle}

\def\eps{\varepsilon}

\def\pL#1#2{{{\rm L}^{#1}(#2)}}

\def\malpha{{\mib \alpha}}

\def\mx{{\bf x}}
\def\mxi{{\mib \xi}}

\def\oi#1#2{\langle#1,#2\rangle}
\def\Pd{{\rm P}}
\def\ph{\varphi}

\def\Rd{{{\bf R}^{d}}}

\def\supp{{\rm supp\,}}

\def\u{{\bf u}}
\def\vv{{\bf v}}

\def\Pd{{\rm P}}

\begin{document}

%\begin{frontmatter}

\title{On a generalization of compensated compactness in the $L^p-L^{q}$ setting}

%\author[label1]{Marin Mi\v sur}
%\address[label1]{University of Zagreb, Faculty of Science, Bijeni\v{c}ka cesta 30,
%10000 Zagreb, Croatia. E-mail: mmisur@math.hr}
%
%
%\author[label2]{Darko Mitrovi\'c}
%\address[label2]{University of Montenegro, Faculty of Natural Sciences and Mathematics, Cetinjski put bb,
%81000 Podgorica, Montenegro. E-mail: matematika@t-com.me}

\author{M.~Mi\v{s}ur}\address{Marin Mi\v{s}ur,
University of Zagreb, Faculty of Science, Bijeni\v{c}ka cesta 30,
10000 Zagreb, Croatia}\email{pfobos@gmail.com}

\author{D.~Mitrovi\'{c}}\address{Darko Mitrovi\'{c},
University of Montenegro, Faculty of Mathematics, Cetinjski put bb,
81000 Podgorica, Montenegro}\email{matematika@t-com.me}

\begin{abstract}
We investigate conditions under which, for two sequences $(\u_r)$
and $(\vv_r)$ weakly converging to $\u$ and $\vv$ in
$L^p(\R^d;\R^N)$ and $L^{q}(\R^d;\R^N)$, respectively, $1/p+1/q \leq
1$, a quadratic form $q(\mx;\u_r,\vv_r)=\sum\limits_{j,m=1}^N q_{j
m}(\mx)u_{j r} v_{m r}$ converges toward $q(\mx;\u,\vv)$ in the
sense of distributions. The conditions involve fractional
derivatives and variable coefficients, and they represent a
generalization of the known compensated compactness theory. The
proofs are accomplished using a recently introduced $H$-distribution
concept. We apply the developed techniques to a nonlinear
(degenerate) parabolic equation.
\end{abstract}

%\begin{keyword}



%\end{keyword}

%\end{frontmatter}

\maketitle

\section{Introduction}

The compensated compactness theory proved to be a very useful tool
in investigating problems involving partial differential equations
(both linear and nonlinear). Suppose, for instance, that we aim to
solve a nonlinear partial differential equation which we write
symbolically as $A[u]=f$, where $A$ denotes a given nonlinear
operator. One of usual approaches is to approximate it by a
collection of {\sl nicer} problems $A_r[u_r]=f_r$, where $(A_r)$ is
a sequence of operators which is somehow close to $A$. Then we try
to prove that the sequence $(u_r)$ converges toward a solution to
the original problem $A[u]=f$. In general, it is relatively easy to
obtain weak convergence on a subsequence of $(u_r)$ towards some
function $u$. Due to the nonlinear nature of $A$, this does not mean
that $u$ will represent a solution to the original problem $A[u]=f$.
However, in some cases, the nonlinearity of $A$ can be {\sl
compensated} by certain properties of the sequence $(u_r)$ (see
\cite{AM, Dpe, Lu} and references therein). The theory which
investigates such phenomena is called compensated compactness and it
was introduced in the works of F.~Murat and L.~Tartar \cite{mur1,
Tar}.

The most general version of the classical result of compensated
compactness theory has been recently proved in \cite{pan_aihp}. Let
us briefly recall it. First, we introduce anisotropic Sobolev spaces
$W^{-1,-2;p}(\R^d)$, where $-1$ is with respect to
$x_1,\dots,x_\nu$ and $-2$ is with respect to $x_{\nu+1},\dots,x_d$, as a subset of tempered distributions

$$
\{ u\in {\cal S}':\; \exists v\in L^p(\Rd), \; k{\hat u}=\hat{v} \},
$$
where $k(\mxi_1,\mxi_2) = \sqrt{1+(2\pi|\mxi_1|)^2 +
(2\pi|\mxi_2|)^4}$, $\mxi_1\in\R^\nu$, $\mxi_2\in\R^{d-\nu}$. It is
H\"ormader's class $B_{p,k}$ and the Banach space with dual
$B_{p',1/k}$ (see chapter 10 of \cite{Hor}). By $\hat{u}$ we denote
the Fourier transform: ${\hat u}(\mxi) = \int_{\R^d}e^{-2\pi i
\mx\cdot\mxi}u(\mx) d\mx$.

Assume that the sequence
$(\u_r)=(u_{1r},\dots,u_{Nr})$ is bounded in $L^p(\R^d;\R^N)$,
$2\leq p \leq\infty$, and converges in ${\cal D}'(\R^d)$ to a
vector function $\u$. Let $q=\frac{p}{p-1}$ if
$p<\infty$, and $q>1$ if $p=\infty$. Assume that the sequences
\begin{equation}
\label{(2)} \sum\limits_{j=1}^N \sum\limits_{k=1}^\nu \pa_{x_k}(a_{s
j k} u_{j r})+ \sum\limits_{j=1}^N \sum\limits_{k,l=\nu+1}^d
\pa_{x_k x_l} (b_{s j kl} u_{j r}),
\end{equation}  for $s=1,\dots,m$, are precompact in the anisotropic Sobolev space
$W^{-1,-2;q}_{loc}(\R^d)$. The (variable) coefficients $a_{s j k}$ and
$b_{s j k l}$ belong to $L^{2\bar{q}}(\R^d)$,
$\bar{q}=\frac{p}{p-2}$ if $p>2$, and to the space $C(\R^d)$ if
$p=2$.

Next, introduce the set
\begin{align}
\label{(3)}
\Lambda(\mx)=\Big\{{\mib\lambda}\in \R^N\big|& \, (\exists \mxi\in \R^d\setminus\{0\})(\forall s=1,\dots,m)\\
& \sum\limits_{j=1}^N\Big(i\sum\limits_{k=1}^\nu a_{s j
k}(\mx)\xi_k-\sum\limits_{k,l=\nu+1}^d b_{s j k l}(\mx)\xi_k\xi_l
\Big)\lambda_{j}=0\ \Big\}. \nonumber
\end{align} Consider the bilinear form on $\R^N$
\begin{equation}
\label{q-form}
q(\mx;{\mib\lambda},{\mib \eta})=Q(\mx){\mib\lambda}\cdot {\mib\eta},
\end{equation}where $Q$ is a symmetric matrix with coefficients
$$q_{j m}\in \begin{cases}
L^{\bar{q}}_{loc}(\R^d), & p>2\\
C(\R^d), & p=2
\end{cases},\quad  j,m=1,\dots,N.$$ Finally, let $q(\mx;\u_r,\u_r)\rightharpoonup \omega$ weakly-$\ast$ in the space of Radon measures.

The following theorem holds
\begin{theorem}
\label{panov} \cite[Theorem 1]{pan_aihp} Assume that
$q(\mx;{\mib\lambda},{\mib\lambda})\geq 0$ for all ${\mib\lambda} \in \Lambda(\mx)$,
a.e. $\mx\in \R^d$. Then $q(\mx;\u(\mx),\u(\mx))\leq \omega$ in the sense of
measures. If  $q(\mx;{\mib\lambda},{\mib\lambda})= 0$ for all ${\mib\lambda} \in
\Lambda(\mx)$, a.e.~$\mx\in \R^d$, then $q(\mx;\u(\mx),\u(\mx))= \omega$.
\end{theorem}  The connection between $q$ and $\Lambda$ given in the previous theorem, we shall call {\em the
consistency condition}.

We would like to formulate and extend the results from Theorem
\ref{panov} to the $L^p-L^{q}$ framework for appropriate (greater
than one) indices $p$ and $q$ where $p<2$. More precisely,
we want to find conditions on two vector-valued sequences $(\u_r)$
and $(\vv_r)$ weakly converging to $\u$ and $\vv$ in $L^p(\R^d)$ and
$L^{q}(\R^d)$, respectively, to ensure that the sequence
$(q(\mx;\u_r,\vv_r))$, where $q$ is the bilinear form from
\eqref{q-form}, satisfies
\begin{equation}
\label{p-p'}
\lim\limits_{r\to \infty} q(\mx;\u_r,\vv_r)=q(\mx;\u,\vv) \ \ {\rm in} \ \ {\cal D}'(\R^d).
\end{equation} Ideally, it should be $1/p+1/q=1$. Due to
technical obstacles (see Remark \ref{rmk_equality}), we are able to
prove \eqref{p-p'} only when $1/p+1/q<1$. However, under additional
assumptions on the sequences $(\u_r)$ and $(\vv_r)$, we are also
able to obtain the optimal $L^p-L^{p'}$-variant of the compensated
compactness. Here and in the sequel, $1/p+1/p'=1$.

This extension will be done in the next section. In the last section we shall
show how to apply this result to a (nonlinear) parabolic type
equation.

\section{The main result}

In order to formulate the $L^p-L^q$ variant of the compensated
compactness, we need $H$-distributions.

%%We would like to formulate them so that they comprise conditions \eqref{(2)}-\eqref{q-form} in the case when $p=p'=2$, and, at the same time, to be applicable when $p <2$. Before we formulate them in the next section, let us introduce the tools that we are going to use -- the $H$-distributions.

They were introduced in \cite{AM} as an extension of the $H$-measure
concept (see \cite{Ger, Tar, ALjfa, LM2} and references therein). Let us recall that $H$-measures describe the
loss of strong precompactness for sequences belonging to $L^p$ for
$p\geq 2$, and they were the basic tool in the mentioned work on
compensated compactness \cite{pan_aihp}. The variant of
$H$-distributions that we are basically going to use is formulated
in \cite{LM3, LM4}. Let us recall its definition.

We need multiplier operators with symbols defined on a manifold
$\Pd$ determined by an $d$-tuple
$\malpha=(\alpha_1,\dots,\alpha_d)\in\R^d_{+}$, where $\alpha_k\in
\N$ or $\alpha_k \geq d$
%where $\R^+$ is the set of positive real numbers:
$$
\Pd=\Big\{\mxi\in \R^d: \; \sum\limits_{k=1}^d
|\xi_k|^{2\alpha_k}=1 \Big\}.
$$ The manifold $\Pd$ is smooth enough and we are able to associate an $L^p$ multiplier to a function defined
on $\Pd$ as follows. We define the projection from $\R^d\backslash
\{0\}$ to $\Pd$ by means of the mapping
$$
\big(\pi_{\Pd}(\mxi)\big)_j=\xi_j\,
\Big(|\xi_1|^{2\alpha_1}+\dots+|\xi_d|^{2\alpha_d} \Big)^{- 1/2
\alpha_j}, \ \ j=1,\dots,d, \ \ \mxi\in \R^d\backslash\{0\}.
$$

%In a similar manner as it is usually shown that Sobolev spaces are separable, the separability of $C^d(\Pd)$ can be proven.

Let us now recall the Marcinkiewicz
multiplier theorem \cite[Theorem 5.2.4.]{Gra}, more precisely its
corollary which we provide here:

\begin{corollary}
\label{m1} Suppose that  $\psi\in \pC{d}{\R^d\backslash
\cup_{j=1}^d\{\xi_j= 0\}}$ is a bounded function such that for some
constant $C>0$ it holds
\begin{equation}
\label{c-mar} |\mxi^{\tilde\malpha} \partial^{\tilde\malpha}
\psi(\mxi)|\leq C,\ \
 \mxi\in \R^d\backslash \cup_{j=1}^d\{\xi_j= 0\}
\end{equation}
for  every multi-index
$\tilde\malpha=(\tilde\alpha_1,\dots,\tilde\alpha_d) \in {\bf
N}_0^d$ such that
$|\tilde\malpha|=\tilde\alpha_1+\tilde\alpha_2+\dots+\tilde\alpha_d
\leq d$. Then, the function $\psi$ is an ${\rm L}^p$-multiplier for
$p\in \oi 1\infty$, and the operator norm of ${\cal A}_\psi$ equals
to $C_{d,p}$, where $C_{d,p}$ depends only on $C$, $p$ and $d$.
\end{corollary}

The following statement holds.

\begin{theorem}
\label{prop_repr} \cite{LM3} Let $(u_n)$ be a bounded sequence in
$L^p (\R^{d})$, $p>1$, and let $(v_n)$ be a bounded sequence of
uniformly compactly supported functions in $L^\infty(\R^d)$ weakly converging to $0$ in
the sense of distributions. Then,
after passing to a subsequence (not relabelled), for any
$\bar{p}\in\langle1,p\rangle$ there exists a continuous bilinear functional $B$
on $L^ {\bar{p}'}({\R^{d}}) \otimes C^{d}(\Pd)$ such that for every
$\varphi\in L^{\bar{p}'}(\R^{d}) $ and $\psi \in C^{d}(\Pd)$ it
holds
\begin{equation}
\label{rev1}
 B(\varphi, \psi)=\lim\limits_{n\to
\infty}\int_{\R^{d}} \varphi(\mx)u_n(\mx) \overline{\big({\cal
A}_{\psi_\Pd} v_n\big)(\mx)} d\mx \,,
\end{equation}
where ${\cal A}_{\psi_\Pd}$ is the (Fourier) multiplier operator on
$\R^d$ associated to $\psi\circ \pi_\Pd$ and $\frac{1}{\bar{p}}+\frac{1}{\bar{p}'}=1$.

The bound of the functional $B$ is equal to $C_u\, C_v \, C_{d,q}$, where $C_u$ is the $L^p$-bound of the sequence $(u_n)$; $C_v$ is the $L^q$-bound of the sequence $(v_n)$ where $\frac{1}{p}+\frac{1}{\bar{p}'}+\frac{1}{q}=1$; and $C_{d,q}$ is the constant from Corollary \ref{m1}.
\end{theorem}

We shall now prove that we can extend the bilinear functional $B$
from the previous theorem to a functional on $L^{\bar{p}'}(\R^{d};
C^d(\Pd))$. We shall need the following theorem a proof of which in the
case of real functionals can be found in \cite{LM3}.

\begin{theorem}
\label{thm1} Let $B$ be a (complex valued) continuous bilinear
functional on $L^{p}(\Rd)\otimes E$, where $E$ is  a separable
Banach space and $p \in \oi 1\infty$. Then $B$  can be extended as a
(complex valued) continuous functional on $L ^{p}(\R^{d};E)$ if and
only if there exists a (nonnegative) function $b\in L^{{p'}}(\Rd)$
such that for every $\psi\in E$ and almost every $\mx\in\Rd$, it
holds
\begin{equation}
\label{cond} | \tilde B \psi(\mx)| \leq b(\mx) \|\psi\|_{E} \,,
\end{equation}
where $\tilde B $ is a bounded linear operator $E \to L^{{p'}}(\Rd)$
defined by $\Dup{\tilde B \psi}\ph =B(\ph, \psi)$, $\varphi \in
L^p(\R^d)$.

\end{theorem}

\begin{proof}

The proof goes along the lines of the proof of \cite[Theorem
2.1]{LM3} when we separately consider real ($\Re$) and imaginary
($\Im$) parts of the functional $B$ and the operator $\tilde{B}$. Let us briefly
recall it.

Let us  assume that \eqref{cond} holds. In order to prove that $B$
can be extended as a linear functional on $\pL{{p}}{\R^{d};E}$, it
is enough to obtain an appropriate bound on the following dense
subspace  of $\pL{{p}}{\R^{d};E}$:
\begin{equation}
\label{spdense} \Big\{\sum\limits_{j=1}^{N} \psi_j\chi_j(\mx): \;
\psi_j\in E,  N\in \N \Big\}\,,
\end{equation} where $\chi_{i}$ are characteristic functions associated to mutually disjoint, finite measure sets.

For an arbitrary function $g=\sum\limits_{i=1}^{N}\psi_i \chi_i$
from  \eqref{spdense}, the bound follows easily once we notice
that
\begin{align*}
&\big| B(\sum\limits_{j=1}^N \psi_j \chi_j) \big|:=\big|
\sum\limits_{j=1}^N B(\chi_j,\psi_j) \big|= \big|\int_{\R^d}
\sum\limits_{j=1}^N \tilde{B}\psi_j(\mx) \chi_j(\mx) d\mx \big|\\&
\leq \int_{\R^d} b(\mx) \sum\limits_{j=1}^N \chi_j(\mx)
\|\psi_j\|_{E}d\mx \leq \|b\|_{L^{p'}(\R^d)} \|g\|_{L^{p}(\R^d;E)}.
\end{align*}

%where $g=\sum\limits_{j=1}^N\chi_j \psi_j$.

%$$
%\|\phi\|^{p}_{\pL{{p}}{\R^{d};E}}=\int_{\R^{d}}
%\|\sum\limits_{i=1}^{N}\psi_i \chi_i(\mx)\|_{E}^{p}d\mx =
%\int_{\R^{d}} \sum\limits_{i=1}^{N} \|\psi_i\|_{E}^{p}\chi_i(\mx)
%d\mx \,,
%$$
In order to prove the converse, take a
countable dense set of functions from the unit ball of $E$, and
denote them by $\psi_j$, $j\in \N$. Assume that the functions
$\psi_{-j}:=-\psi_j$ are also in $E$. For each function $\tilde
B\psi_j \in \pL {p'}\Rd$ denote by $D_j$ the corresponding set of
Lebesgue points, and their intersection by $D=\cap_j D_j$.

For any $\mx\in D$ and $k\in\N$ denote
\begin{equation*}
b_k(\mx)=\max\limits_{|j| \leq k} \Re(\tilde B\psi_j) (\mx)
=\sum\limits_{|j|=1}^k  \Re(\tilde B \psi_j)(\mx) \chi_j^k (\mx)\,
\end{equation*} where $\chi^k_{j_0}$ is the characteristic function of set
$X^k_{j_0}$ of all points $\mx\in D$ for which the above maximum is
achieved for $j=j_0$. Furthermore, we can assume that for each $k$
the sets $X^k_{j}$ are mutually disjoint.
 The sequence $(b_k)$ is clearly monotonic sequence of positive
functions,  bounded in $\pL {p'}\Rd$, whose limit (in the same
space) we denote by $b^{\Re}$. Indeed, choose $\varphi \in
L^p(\R^d)$, $g=\sum\limits_{|j|=1}^k \varphi(\mx)\chi_j^k(\mx)
\psi_j \in L^p(\R^d;E)$, and consider:
\begin{align*}
&\int_{\R^d}b_k(\mx) \varphi(\mx)d\mx=\Re \big( \int_{\R^d}\tilde{B} \sum\limits_{|j|=1}^k \psi_j \chi_j^k(\mx) \varphi(\mx)d\mx \big)\\
&=\Re\big( \sum\limits_{|j|=1}^k B(\chi_j^k
\varphi,\psi_j)\big)=\Re\big( B(g) \big)\leq C \|g\|_{L^p(\R^d;E)}
\leq C \|\varphi\|_{L^p(\R^d)},
\end{align*} where $C$ is the norm of $B$ on $(L^p(\R^d;E))'$. Since $\varphi\in L^p(\R^d)$ is arbitrary, we get that $(b_k)$ is bounded in $L^{p'}(\R^d)$.

As $D$ is a set of full measure, for every $\psi_j$ we have
$$
| \Re(\tilde B  \psi_j) (\mx) |\leq b^{\Re}(\mx), \quad \ae {\mx\in
\Rd}.
$$ We are able to obtain a similar bound for the imaginary part of
$\tilde{B}\psi_j$. In other words, there exists $b^{\Im}\in \pL
{p'}\Rd$ such that
$$
| \Im(\tilde B  \psi_j) (\mx) |\leq b^{\Im}(\mx), \quad \ae {\mx\in
\Rd}.
$$

The assertion now follows since \eqref{cond} holds for
$b=b^{\Re}+b^{\Im}$ on the dense set of functions $\psi_j$, $j\in
\N$. For details see \eqref{dense} below. \end{proof}

We need the following lemma which will also be used in the
last section.

\begin{lemma}
\label{r-c} If the real symbol $\psi \in C^d(\Pd)$ of the multiplier
operator ${\cal A}_\psi$ is an even function
($\psi(\mxi)=\psi(-\mxi)$), then for every real $u\in L^p(\R^d)$,
$p>1$, ${\cal A}_\psi(u)$ is a real function for a.e. $\mx\in \R^d$.

If the real symbol $\psi \in C^d(\Pd)$ of the multiplier operator
${\cal A}_\psi$ is an odd function ($\psi(\mxi)=-\psi(-\mxi)$), then
for every real $u\in L^p(\R^d)$, $p>1$, ${\cal A}_\psi (u)$ is a
purely imaginary function for a.e. $\mx\in \R^d$.
\end{lemma}

\begin{proof} Assume first that the symbol $\psi$ is an even function. It is
enough to prove that, for arbitrary real functions $u,v \in
L^2(\R^d) \cap L^p(\R^d)$, it holds
\begin{align*}
\int v {\cal A}_{\psi}(u) d\mx=\int v \overline{{\cal A}_{\psi}(u)}
d\mx.
\end{align*} This follows from the Plancherel theorem, and the
change of variables $\mxi \mapsto -\mxi$. Indeed,

\begin{align*}
\int v {\cal A}_{\psi}(u) d\mx &=\int \overline{v} {\cal
A}_{\psi}(u) d\mx=\int \psi(\mxi) \overline{\hat{v}}(\mxi)
\hat{u}(\mxi)d\mxi=\left(\mxi\mapsto -\mxi \right)\\
&=\int \psi(\mxi) {\hat{v}}(\mxi) \overline{\hat{u}}(\mxi)d\mxi=\int
v \overline{{\cal A}_{\psi}(u)} d\mx.
\end{align*}

The proof is the same when the symbol is odd. \end{proof}

Now, we can prove the following proposition.

\begin{proposition} \cite{LM3}
\label{korolar} The bilinear functional $B$ defined in Theorem
\ref{prop_repr} can be extended by continuity to a functional on
$L^{\bar{p}'}(\R^{d}; C^d(\Pd))$. The bound of the extension is
equal to the bound of the bilinear functional $B$ (with the
notations of Theorem \ref{prop_repr}, it is $C_u \,C_v \,C_{d,q}$,
$1/p+1/\bar{p}'+1/q=1$).
\end{proposition}

\begin{remark}
The proof of the proposition can also be found in \cite{LM4}. Since
this paper is still unpublished, we give a slightly different proof
here.
\end{remark}

\begin{proof} We will show that $B$ satisfies conditions of Theorem \ref{thm1}, namely, that there exists a function $b\in L^{\bar
p}(\Rd)$ such that for every $\psi \in C^d(\Pd)$,
$\|\psi\|_{C^d(\Pd)}\leq 1$ and almost every $\mx\in\Rd$ it holds

\begin{equation}
\label{ocjena}
|(\tilde{B}\psi)(\mx)|\leq b(\mx) \|\psi\|_{C^d(\Pd)},
\end{equation}

\noindent where $\tilde B:C^d(\Pd) \to L^{\bar p}(\Rd)$ is a bounded
linear operator defined by $\langle {\tilde B}\psi, \varphi\rangle =
B(\varphi,\psi)$, $\varphi\in L^{\bar{p}'}(\R^{d})$.

%Note that if the set $L:=\{ \psi \in
%C^d(\Pd)$: $\|\psi\|_{C^d(\Pd)}\leq 1 \}$ were at most countable, we
%could define $b\in L^{\bar s}(\Rd)$ in the following straightforward
%way
%$$
%b(\mx) = {\rm sup}_{\psi\in K}|(\tilde{B}\psi)(\mx)|.
%$$
%
%However, $L$ is uncountable, so
%this definition does not necessarily result in a measurable function
%$b$.

We proceed as follows: choose a dense countable set $E$ of functions
$\psi_j$, $j\in \N$, from the set $\{ \psi \in C^d(\Pd)$:
$\|\psi\|_{C^d(\Pd)}\leq 1 \}$. Define functions
$\psi_{-j}(\mxi)=-\psi_j(\mxi)$ and add them to $E$. Moreover, add the
linear combinations of the form
$\psi_j^e(\mxi)=\frac{1}{2}(\psi_j(\mxi)+\psi_{j}(-\mxi))$ and
$\psi_j^o(\mxi)=\frac{1}{2}(\psi_j(\mxi)-\psi_{j}(-\mxi))$ for $j\in {\bf Z}\setminus\{0\}$ to $E$ as well. Remark
that functions $\psi_j^e$ are even, while $\psi_j^o$ are odd
(in the sense of Lemma \ref{r-c}) and that the set $E$ is still countable and dense.

For each $j$ choose a function $\tilde B\psi_j$ from $L^{\bar{p}}(\R^d)$ and
denote by $D_j$ the corresponding set of Lebesgue points (for definiteness, we can take $\tilde B\psi_j$
to be the precise representative of the class (see chapter 1.7. of \cite{Evans})). The set
$D_j$ is of full measure, and thus the set $D=\cap_j D_j$ as well.

 For any $\mx\in D$ and $k\in\N$ denote ($i=\sqrt{-1}$ below)
\begin{align}
\label{bk1} b^e_k(\mx):&=\max\limits_{|j| \leq k} \tilde B\psi^e_j
(\mx) =\sum\limits_{|j|=1}^k  \tilde B \psi^e_j (\mx) \chi_j^k
(\mx) \in \R^+,\\
\label{bk1o} b^o_k(\mx):&=\max\limits_{|j| \leq k} i \tilde
B\psi^o_j (\mx) =\sum\limits_{|j|=1}^k  i \tilde B \psi^o_j (\mx)
\chi_j^k (\mx)\in \R^+,
\end{align} where $\chi^k_{j_0}$ is a characteristic function of the set of all
points  for which the above maximum is achieved for $\psi^e_{j_0}$ ($\psi^o_{j_0}$ respectively) and it has not been achieved
for $\psi^e_j$ ($\psi^o_{j}$ respectively), $-k\leq j < j_0$.

First, note that we can make sure that $\chi^k_j$ have disjoint
supports for fixed $k$: define $\chi_j^k$ to be equal to one on the
set
$$
\Big\{ \mx\in D: (\tilde{B}\psi^e_j)(\mx) = b^e_k(\mx)\ \&\ (\forall l < j) (\tilde{B}\psi^e_l)(\mx) < b^e_k(\mx) \Big\},
$$ and extend it with zero to the whole $\Rd$.

Next, we shall prove that the sequence of functions $(b^e_k)$ is
bounded in $L^{\bar{p}}(\Rd)$. To this effect, take an arbitrary
$\phi\in C_c(\R^{d})$, and denote $K=\supp \phi$. Since $(v_n)$ is a
bounded sequence of uniformly compactly supported functions in
$L^\infty(\R^d)$, it belongs to $L^q(\R^d)$ for every $q\in\langle
1,\infty\rangle$. Since $\bar{p} < p$, we can find $q > 1$ such that
$1/q + 1/{{\bar p}'} = 1/{p'}$. Fix such $q$. Choose $r>1$ such that
$q=r' p'$. Denote by $\chi^{k, \eps}_{j}\in C_c(\R^{d})$,
$j=1,\dots,k$ smooth approximations of characteristic functions from
\eqref{bk1} on $K$ such that (note that $\|\chi^{k}_{j}\|_{L^\infty}\leq 1$)

$$
\|\chi^{k, \eps}_{j}-\chi^k_{j} \|_{L^{\max\{p',r\}}(K)}\leq
\frac{\eps}{2k}.
$$
%or
%$$\|\chi^{k, \eps}_{j}-\chi^k_{j} \|_{L^{s'}(K)}\leq \frac{\eps}{k} \hbox{ if $s' > r$}$$
As before, denote by $C_u$ an $L^
p$ bound of $(u_n)$ and by $C_{v}$ an $L^{q}$ bound of $(v_n)$ .

According to \eqref{bk1} and the definition of operator $\tilde
B$, we have
\begin{align*}
&\big|_{L^{\bar{p}}(\R^d)}\langle b^e_k, \phi
\rangle_{L^{\bar{p}'}(\R^d)}\big| =\Big| \lim\limits_{n\to
\infty}\int_{\R^{d}} \sum\limits_{|j|=1}^k ({\phi} u_n
\chi^k_j) (\mx)( \overline{{\cal A}_{\psi^e_j} v_n)(\mx)} d\mx \Big|
\\&\leq \limsup\limits_{n\to \infty} \int_{\R^{d}}
\left(\sum\limits_{|j|=1}^k |u_n|^p \chi^k_j  \, (\mx) \right)^{1/p}
\left( \sum\limits_{|j|=1}^k \chi_j^k |\phi\, {\cal A}_{\psi^e_j}
v_n|^{p'}\, (\mx)\right)^{1/p'} d\mx
\\&\leq \limsup\limits_{n\to \infty}\Big\|\sum\limits_{|j|=1}^k |u_n|^p
\chi^k_j\Big\|^{1/p}_{L^1(\R^{d})} \Big\|\sum\limits_{|j|=1}^k
\chi^k_j |\phi\, {\cal
A}_{\psi^e_j}v_n|^{p'}\Big\|^{1/p'}_{L^1(\R^{d})} \\
&\leq\limsup\limits_{n\to \infty} \|u_n\|_{L^p(\R^{d})} \,
\Big(\Big\| \sum\limits_{|j|=1}^k (\chi^k_j-\chi^{k, \eps}_j)
|\phi\, {\cal
A}_{\psi^e_j}v_n|^{p'}\Big\|_{L^{1}(\R^{d})} \nonumber\\
&\qquad\qquad\qquad\qquad\qquad\qquad+ \sum\limits_{|j|=1}^k\Big\|
\chi^{k, \eps}_j |\phi\, {\cal
A}_{\psi^e_j}v_n|^{p'}\Big\|_{L^{1}(\R^{d})} \Big)^{1/p'}\nonumber
\end{align*}
\begin{align*}&\leq\limsup\limits_{n\to \infty} \|u_n\|_{L^p(\R^{d})} \,
\Big(\Big\| \sum\limits_{|j|=1}^k (\chi^k_j-\chi^{k, \eps}_j)
|\phi\, {\cal
A}_{\psi^e_j}v_n|^{p'}\Big\|_{L^{1}(\R^{d})} \nonumber\\
&\qquad\qquad\qquad\qquad\qquad\qquad+
\sum\limits_{|j|=1}^k\Big\| \chi^{k, \eps}_j \phi\, {\cal A}_{\psi^e_j}v_n\Big\|^{p'}_{L^{p'}(\R^{d})} \Big)^{1/p'}\\
&\leq  C_u \limsup\limits_{n\to \infty}\Big(\sum\limits_{|j|=1}^k \|
\chi^k_j-\chi^{k, \eps}_j \|_{L^{r}(K)}  \| {\cal A}_{\psi^e_j}(\phi
v_n )\|^{p'}_{L^
{q}(\R^{d})}\\&\qquad\qquad\qquad\qquad\qquad\qquad+\sum\limits_{|j|=1}^k\|{\cal
A}_{\psi^e_j}(\chi^{k, \eps}_j \phi v_n)\|^{p'}_{L^{p'}(\R^{d})}
\Big)^{1/p'},
\end{align*} where in the second step we have used discrete version of H\"{o}lder
inequality and the fact that $|\lim_n a_n| \leq \limsup_n |a_n|$; in the last step we have used a version of the first
commutation lemma \cite[Lemma 3.1]{AM} (see also \cite[Lemma 2]{LM4}) and H\"{o}lder inequality with $1/r + 1/{r'} =
1$ remembering that $r' p'=q$. By means of Corollary 2 and properties of the functions $\chi^{k, \eps}_j$ it
follows
\begin{align*}
&\big|\langle b^e_k, \phi \rangle\big| \leq C_u \limsup
\limits_{n\to \infty}\left(\eps \, C_\phi \, C^{p'}_{q,d}\,\|v_n
\|^{p'}_{L^ {q}(\Rd)}\!+\!C^{p'}_{p',d}\,\sum\limits_{|j|=1}^k
\|\chi^{k, \eps}_j \phi v_n\|^{p'}_{L^
{p'}(\R^{d})}\!\right)^{\!1/p'}\hskip -3mm,
\end{align*} where $C_{p',d}$ is the constant from Corollary 2 (recall that
$\|\psi^e_j\|_{C^d(\Pd)} \leq 1$), while
$C_\phi=\|\phi \|^{p'}_{L^\infty(\Rd)}$.
By letting  $\eps
\to 0$, we conclude
\begin{align*}
\big|\langle b^e_k, \phi \rangle\big| \leq C_u C_{p',d}\limsup
\limits_{n\to \infty} \left(\sum\limits_{|j|=1}^k \|\chi^{k}_j \phi
v_n\|^{p'}_{L^ {p'}(\R^{d})}\right)^{1/p'}
\end{align*} since $\chi^{k, \eps}_j\to \chi^k_j$ in $L^ {p'}(K)$.
Since supports of functions $\chi^k_j$ are disjoint and remembering the choice of $q$, we get
$$
\sum\limits_{|j|=1}^k \|\chi^{k}_j \phi v_n\|^{p'}_{L^ {p'}(\R^{d})}
\leq \|\phi v_n\|^{p'}_{L^ {p'}(\R^{d})}\leq
\left(\|\phi\|_{L^{\bar{p}'}(\Rd)} \|v_n\|_{L^ {q}(\Rd)}\right)^{
p'}\, ,
$$
since $\sum\limits_{|j|=1}^k (\chi^{k}_j)^{p'} =
\sum\limits_{|j|=1}^k \chi^{k}_j \leq 1$. From this, it follows
\begin{equation*}
\big|\langle b^e_k, \phi \rangle\big| \leq
C_u C_{d,p'}C_{v} \|\phi\|_{L^{\bar{p}'}(\Rd)},
\end{equation*} where all the constants on the right hand side do not depend on $k$.
Since $C_c(\Rd)$ is dense in $L^{\bar{p}'}(\Rd)$ we conclude that
the sequence $(b^e_k)$ is bounded in $L^{\bar{p}}(\Rd)$. Noticing that
$(b^e_k)$ is a non-decreasing sequence of positive functions, it
follows from Beppo-Levi's theorem on monotone convergence that its
(pointwise) limit  $b^e$ is an $L^{\bar{p}}(\Rd)$ function.

In the completely same way, we conclude that $(b^o_k)$ converges
toward $b^o \in L^{\bar{p}}(\Rd)$.

The function $b=b^e+b^o$ satisfies \eqref{ocjena} for
$\tilde{B}\psi$ when $\psi=\psi^e_j+\psi^o_{j'}$ for some $j,j'\in
{\bf Z}\setminus\{0\}$. On the other hand, every $\psi\in C^d(\Pd)$ can be represented
as a sum of odd and even functions as follows
$\psi(\mxi)=\frac{1}{2}(\psi(\mxi)+\psi(-\mxi))+\frac{1}{2}(\psi(\mxi)-\psi(-\mxi))$
and we conclude that \eqref{ocjena} holds for any $\psi \in E$. By
continuity, the statement can be generalised to an arbitrary
$\psi\in C^d(\Pd)$: take a sequence $(\psi_n)\subseteq E$ such that
$\psi_n \to \psi$ in $C^d(\Pd)$ and write

\begin{align}
\label{dense} \int_\Rd |(\tilde{B}\psi)(\mx)|\varphi(\mx) d\mx &\leq
\int_\Rd |(\tilde{B}\psi - \tilde{B}\psi_n)(\mx)|\varphi(\mx) d\mx +
\int_\Rd|(\tilde{B}\psi_n)(\mx)|\varphi(\mx) d\mx\\
&\leq o_n(1) + \int_\Rd b(\mx)\varphi(\mx) d\mx, \nonumber
\end{align}for arbitrary $\varphi \in C^\infty_c(\Rd;\R^+_0)$ where we have used continuity of $\tilde{B}$.
Due to arbitrariness of the function $\varphi$, the result follows
from Theorem \ref{thm1}. \end{proof}

%\begin{equation}
%\label{min1} \tilde{b}_k(\mx):=\min\limits_{1\leq j \leq k} \tilde
%B\psi_j (\mx) =\sum\limits_{j=1}^k  \tilde B \psi_j (\mx) \chi_j^k
%(\mx),
%\end{equation} where $\chi^k_{j_0}$ is a characteristic function of the set of all
%points  for which the above minimum is achieved for $\psi_{j_0}$,
%admits the limit $b_{min}\in L^{\bar{s}}(\Rd)$.
%
%The function $b(\mx)=max\{|b_{max}(\mx)|,|b_{min}(\mx)| \}\in
%L^{\bar{s}}(\Rd)$ satisfies \eqref{ocjena}. The result now follows
%from Theorem 2.1 from \cite{LM3}.

\begin{remark} Note that if the set $L:=\{ \psi \in
C^d(\Pd)$: $\|\psi\|_{C^d(\Pd)}\leq 1 \}$ were at most countable, we
could have defined $b\in L^{\bar p}(\Rd)$ in the following straightforward
way
$$b(\mx) = {\rm sup}_{\psi\in L}|(\tilde{B}\psi)(\mx)|.$$

However, $L$ is uncountable, so this definition does not necessarily
result in a measurable function. Taking supremum over a countable
dense subset of $L$ would result in a measurable function which may
not be $L^{\bar{p}}$-function.
\end{remark}

Now, we are ready to prove a variant of compensated compactness in
the $L^p-L^q$ framework. Before we proceed, we recall that the dual
of the space $L^p(\R^d;C^d(\Pd))$ is the space
$L^{p'}_{w*}(\R^d;C^d(\Pd)')$ of weakly-$\ast$ measurable functions
$B:\R^d\to C^d(\Pd)'$ such that $\int_{\R^d}\|B(\mx)
\|^{p'}_{C^d(\Pd)'} d\mx$ is finite (for details see \cite[p.
606]{edv}).

We first need to extend the notion of
$H$-distributions from Theorem \ref{prop_repr} as follows.

\begin{theorem}
\label{prop_repr1} Let $(u_r)$ be a sequence of uniformly compactly
supported functions weakly converging to zero in $L^p (\R^{d})$,
$p>1$, and let $(v_r)$ be a bounded sequence of uniformly compactly
supported functions in $L^{q}(\R^d)$, $1/q+1/p<1$, weakly converging
to $0$ in the sense of distributions. Then, after passing to a
subsequence (not relabelled), for any $\bar{p}\in
\langle1,\frac{pq}{p+q}\rangle$ there exists a continuous bilinear
functional $B$ on $L^ {\bar{p}'}({\R^{d}}) \otimes C^{d}(\Pd)$ such
that for every $\varphi\in L^{\bar{p}'}(\R^{d}) $ and $\psi \in
C^{d}(\Pd)$, it holds
\begin{equation}
\label{rev2}
 B(\varphi, \psi)=\lim\limits_{r\to
\infty}\int_{\R^{d}} \varphi(\mx)u_r(\mx) \overline{\big({\cal
A}_{\psi_\Pd} v_r\big)(\mx)} d\mx \,,
\end{equation}
where ${\cal A}_{\psi_\Pd}$ is the (Fourier) multiplier operator on $\R^d$
associated to $\psi\circ \pi_\Pd$.

The bilinear functional $B$ can be continuously extended to a linear functional on $L^{\bar{p}'}(\R^d;C^d(\Pd))$.
\end{theorem}
\begin{proof}
%Denote by $u$ a weak limit of the sequence $(u_n)$
%along a not relabeled subsequence. Also, redenote by $(u_n)$ the
%sequence $(u_n-u)$ where $(u_n)$ is previously chosen subsequence.
%Thus, in the sequel, we assume that $u_n \rightharpoonup 0$ as $n\to
%\infty$.

Introduce the truncation operator
\begin{equation}
\label{trunc}
T_l(v)=\begin{cases}
v, & |v|< l\\
0, & |v|\geq l
\end{cases}, \ \ l\in \N,
\end{equation}and rewrite $v_r$ in the form
$$
v_r(\mx)=T_l(v_r)(\mx)+(v_r-T_l(v_r))(\mx),
$$ where $T_l(v_r)$ is understood pointwisely. Notice that
\begin{equation}
\label{uni_lim} \limsup\limits_{l,r\to
\infty}\|v_r-T_l(v_r)\|_{L^1(K)}= 0
\end{equation} for any relatively compact measurable $K\subseteq \R^d$. Indeed, denote by
$$
\Omega_r^l=\{ \mx\in\R^d:\, |v_r(\mx)| > l \}.
$$It holds
\begin{equation}
\label{conv0} \lim\limits_{l\to \infty} \sup\limits_{r\in \N} {\rm
meas}(\Omega_r^l) =0.
\end{equation}The latter follows since $(v_{r})$ is bounded in $L^q(\R^d)$ and
\begin{align*}
&\sup\limits_{r\in \N} \int_{\Rd }|v_{r}(\mx)|^q d\mx
\geq \sup\limits_{r\in \N} \int_{\Omega_r^l} l^q d\mx
\geq l^q  \sup\limits_{r\in \N}{\rm
meas}(\Omega_r^l).
\end{align*} Now, we simply use the H\"{o}lder inequality
$$
\int_K |v_{r}-T_l(v_{r})|dx=\int_{K \cap \Omega_r^l} |v_{r}|dx \leq {\rm meas}(K \cap \Omega_r^l)^{1/q'} \|v_{r}\|_{L^{q}(K)}
$$ and this tends to zero uniformly with respect to $r$ and $l$
according to \eqref{conv0} and the boundedness of $(v_{r})$ in
$L^q(\R^d)$. Thus, \eqref{uni_lim} is proved. Since $(v_r)$, and
therefore $(T_l(v_r))$ are bounded in $L^q(\R^d)$, \eqref{uni_lim}
and interpolation inequalities imply that for any $\bar{q}\in
[1,q\rangle$
\begin{equation}
\label{ruj1} \limsup\limits_{l,r\to
\infty}\|v_r-T_l(v_r)\|_{L^{\bar{q}}(K)}= 0.
\end{equation} Next, denote by $\mu_l$ the $H$-distribution corresponding
to $(u_r)$ and $(T_l(v_r))$ in the sense of Theorem \ref{prop_repr}.
From here and \eqref{uni_lim}, we conclude that we can rewrite the
right-hand side of \eqref{rev2} in the form
\begin{align}
\label{zg1} &\lim\limits_{r\to \infty}\int_{\R^{d}}
\varphi(\mx)u_r(\mx) \overline{\big({\cal A}_{\psi_\Pd}
v_r\big)(\mx)} d\mx\\&\!=\! \lim\limits_{r\to \infty}\left(
\int_{\R^{d}} \varphi(\mx)u_r(\mx) \overline{{\cal A}_{\psi_\Pd}
\big( T_l(v_r)\big)(\mx)} d\mx\!+\!\int_{\R^{d}}
\varphi(\mx)u_r(\mx) \overline{{\cal A}_{\psi_\Pd}
\big(v_r\!-\!T_l(v_r)\big)(\mx)}
d\mx\right)\nonumber\\&=\langle\mu_l,\varphi \psi \rangle+o_{l}(1),
\nonumber
\end{align}where $o_{l}(1)\to 0$ as $l\to \infty$ follows from \eqref{ruj1} and the application of the H\"{o}lder inequality as
follows:
\begin{align*}
&|\int_{\R^{d}} \varphi(\mx)u_r(\mx) \overline{{\cal A}_{\psi_\Pd}
\big(v_r\!-\!T_l(v_r)\big)(\mx)} d\mx| \\&\leq C_{d,\bar{q}}\|
\varphi \|_{L^{\bar{p}'}(\R^d)} \, \|\psi\|_{C^d(P)} \,
\sup\limits_r\| u_r \|_{L^p(\R^d)} \, \sup\limits_r\|
v_r\!-\!T_l(v_r) \|_{L^{\bar{q}}(\R^d)},
\end{align*} where $1/\bar{p}'+1/p+1/\bar{q}=1$ (and obviously $\bar{q}<q$
implying that we can apply \eqref{ruj1}).

Since $\psi\circ \pi_\Pd$ is an $L^{\bar{q}}$-multiplier
(\cite[Lemma 5]{LM2}), by the H\"older inequality used with the
exponents $\bar{p}'$, $p$, and $\bar{q}<q$, we get

\begin{align}
\nonumber \big|\int_{\R^{d}} \varphi(\mx)u_r(\mx)
\overline{\big({\cal A}_{\psi_\Pd} T_l(v_r)\big)(\mx)} d\mx\big|
&\leq C_{d,\bar{q}}\|\varphi\|_{L^{\bar{p}'}(\R^d)}\|u_r\|_{L^p(\R^d)}\|\psi\|_{C^d(P)}\|T_l(v_r)\|_{L^{\bar{q}}(\R^d)} \\
 &\leq  C_u\,
 C_v\,C_{d,\bar{q}}\|\varphi\|_{L^{\bar{p}'}(\R^d)}\|\psi\|_{C^d(P)}
 \nonumber
\end{align} From here, after passing to the limit $r\to\infty$ and using the
continuity of extension from Proposition \ref{korolar}, we conclude
that $(\mu_l)$ is bounded sequence in $(L^{\bar{p}'}(\R^d;C^d(\Pd))'
=L_{w*}^{\bar{p}}(\R^d;C^d(\Pd)')$ (remark that the bound of
$(\mu_l)$ is $C_u\,C_v\,C_{d,\bar{q}}$). Since
$L_{w*}^{\bar{p}}(\R^d;C^d(\Pd)')$ is dual of the Banach space,
according to the Banach-Alaoglu theorem, $(\mu_l)$ admits a
weak-$\ast$ limit $\mu\in L_{w*}^{\bar{p}}(\R^d;C^d(\Pd)')$ along a
subsequence. The functional $\mu$ satisfies \eqref{rev2}.
\end{proof}

\begin{remark}
\label{rmk_equality} In the case $1/p + 1/q = 1$, the same
proof gives us continuous bilinear functional on $C(\R^d)\otimes
C^d(\Pd)$. We cannot use Proposition \ref{korolar} anymore, but using
Schwartz's kernel theorem, we can (only) extend it to a distribution
from ${\cal D}'(\R^d\times\Pd)$. Therefore, our variant of the
compensated compactness is confined on $L^p-L^q$ framework for
$1/p+1/q<1$. However, under additional assumptions, we are able to
prove the result in the optimal case $1/p + 1/q = 1$ (Corollary
\ref{optimal}).
\end{remark}

Before we proceed, let us recall the definition of fractional
derivatives. For $\alpha\in\R^{+}$, we define
$\partial^\alpha_{x_k}$ to be a pseudodifferential operator with a
polyhomogeneous symbol $(2\pi i \xi_k)^\alpha$, i.e.
$$\partial^\alpha_{x_k}u = ((2\pi i \xi_k)^\alpha
\hat{u}(\mxi))\check\;.$$

In the sequel, we shall assume that sequences $(\u_r)$ and $(\vv_r)$
are uniformly compactly supported. This assumption can be removed if
the orders of derivatives $(\alpha_1,\dots,\alpha_d)$ are natural
numbers. Otherwise, since the Leibnitz rule does not hold for
fractional derivatives, the former assumption seems necessary.

Let us now introduce the localisation principle corresponding to an $H$-distribution.

\begin{proposition}
\label{localisation}

Assume that sequences $(\u_r)$ and $(\vv_r)$ are bounded in
$L^p(\R^d;\R^N)$ and $L^q(\R^d;\R^N)$, where $1/p + 1/q < 1$, and
converge toward ${\mib 0}$ and $\vv=(v_1,\dots,v_N)$ in the sense of
distributions.

Furthermore, assume that the sequence $(\u_r)$ satisfies, for every $s
= 1,\dots,M$:
\begin{equation}
\label{fract} G_{rs}:= \sum\limits_{j=1}^{N}\sum\limits_{k = 1}^{d}
\partial^{\alpha_k}_{x_k}
(a_{s j k}u_{j r}) \to 0 \text{ in }
W^{-\alpha_1,\dots,-\alpha_d;p}(\R^d),
\end{equation} where $\alpha_k\in \N$ or $\alpha_k>d$, $k=1,\dots,d$, and $a_{s j
k}\in L^{\bar{s}'}(\R^d)$, $\bar{s}\in \langle 1,\frac{pq}{p+q}\rangle$.

Finally, by $\mu_{j m}$ denote the $H$-distribution (Theorem
\ref{prop_repr1}) corresponding to a pair of subsequences of $(u_{j r})$ and
$(v_{m r}-v_m)$. Then the following relations hold in the sense of distributions for $m=1,\dots,N$, $s =
1,\dots,M$ ($i=\sqrt{-1}$ below)

\begin{equation}
\label{vv2} \sum\limits_{j = 1}^N\sum\limits_{k = 1}^n a_{s j k}
(2\pi i \xi_k)^{\alpha_k}\mu_{j m} = 0.
\end{equation}

\end{proposition}
\begin{proof}
Assume, without loosing any generality, that $\vv={\mib 0}$.
Denote by ${\cal B}_\psi$ the Fourier multiplier operator with the symbol

$$
(\psi\circ\pi_\Pd)(\mxi)\frac{(1-\theta(\mxi))}{\left(|\xi_1|^{2\alpha_1}+\dots+|\xi_d|^{2\alpha_d}\right)^{1/2}},
$$ where $\theta$ is a cutoff function equal to one in a neighborhood of zero.

According to \cite[Lemma 5]{LM2}, for any $\psi\in C^d(\Pd)$ and any
$\hat{s}>1$, the multiplier operator ${\cal B}_{\psi}: L^2(\R^d)\cap L^{\hat{s}}(\R^d) \to
W^{\alpha_1,\dots,\alpha_d;\hat{s}}(\R^d)$
is bounded (with $L^{\hat{s}}$ norm considered on the domain of ${\cal B}_\psi$); notice that the symbol of $\partial_{x_k}^{\alpha_k}\circ\cal{B}_\psi$ given by

$$
(\psi\circ\pi_\Pd)(\mxi)\frac{(1-\theta(\mxi))(2\pi i\xi_k)^{\alpha_k}}{\left(|\xi_1|^{2\alpha_1}+\dots+|\xi_d|^{2\alpha_d}\right)^{1/2}},
$$

\noindent is a smooth, bounded function satisfying conditions of Marcinkiewicz's multiplier theorem (\cite[Theorem IV.6.6']{stein} or Corollary 2 here).

Insert in \eqref{fract} the test function $g_{rm}$ given by:
\begin{equation}
\label{jul0218} g_{rm}(\mx)={\cal B}_{\psi}\bigl(\phi v_{m r} \bigr)(\mx), \ \ m\in \{1,\dots, N\}
\end{equation}
where  $\psi\in \pC d\Pd$ and $\phi\in \Cbc\Rd$. We get
\begin{align}
\label{zg2}
\int_{\R^d} G_{rs} \overline{g_{rm}} d\mx &=\int_{\R^d}\sum_{j = 1}^{N}\sum_{k = 1}^{n}
a_{sj k}u_{j r} \overline{{\cal A}_{(\psi\circ\pi_\Pd)(\mxi)\frac{(1-\theta(\mxi))
(2\pi i\xi_k)^{\alpha_k}}{\left(|\xi_1|^{2\alpha_1}+\dots+|\xi_d|^{2\alpha_d}\right)^{1/2}}}(\phi v_{m r})} d\mx\\
&\nonumber= \int_{\R^d}\sum_{j = 1}^{N}\sum_{k = 1}^{n} a_{sj k}u_{j
r} \overline{{\cal A}_{(\psi\circ\pi_\Pd)(\mxi)
\frac{(2\pi i\xi_k)^{\alpha_k}}{\left(|\xi_1|^{2\alpha_1}+\dots+|\xi_d|^{2\alpha_d}\right)^{1/2}}}(\phi v_{m r})} d\mx \ \\
&\nonumber\quad - \int_{\R^d}\sum_{j = 1}^{N}\sum_{k = 1}^{n} a_{sj
k}u_{j r} \overline{{\cal A}_{(\psi\circ\pi_\Pd)(\mxi)
\frac{\theta(\mxi)(2\pi
i\xi_k)^{\alpha_k}}{\left(|\xi_1|^{2\alpha_1}+\dots+|\xi_d|^{2\alpha_d}\right)^{1/2}}}(\phi
v_{m r})} d\mx.
\end{align}
Due to the boundedness properties of operator ${\cal B}_\psi$
mentioned above and the compact support of $\phi$, the sequence $(g_{rm})$ is bounded in $W^{\alpha_1,\dots,\alpha_d;
t}(\R^d)$ for $t\in \langle1,q]$. Letting $r\to \infty$ in \eqref{zg2}, we get \eqref{vv2}
after taking into account Theorem \ref{prop_repr1} and the strong
convergence of $(G_{rs})$. Note that the second summand in the above
identity goes to $0$ because of the compact support of the function
$\theta$.
\end{proof}

%\begin{remark}
%If all $\alpha_k$ are nonnegative integers, we could use the space
%$W^{-1,p}_{loc}(\R^d)$ instead of $W^{-1,p}(\R^d)$ in \eqref{fract}.
%Then the analogous relation to \eqref{fract} holds for any sequence
%$\rho u_{jr}$, where $\rho\in C^\infty_c(\R^d)$. A short proof of
%this claim can be found at the beginning of Theorem 4 in
%\cite{ALjfa}.
%\end{remark}

\begin{remark}

\label{rmk_equality2} In the case $1/p + 1/q = 1$, taking into
account Remark \ref{rmk_equality} and coefficients $a_{sjk}$ from
the space $C_0(\R^d)$, we get the same result as in \eqref{vv2} for
distributions $\mu_{jm}$ from ${\cal D}'(\R^d \times \Pd)$.
\end{remark}

We can now formulate conditions under which \eqref{p-p'} holds.  We
call them the strong consistency conditions. They represent a
generalization of the standard consistency conditions given above.

As before, let $\bar{s} \in \langle 1,\frac{pq}{p+q}\rangle$ be a
fixed number for given $p,q>1$. Introduce the set
\begin{align}
\label{strong-c} \Lambda_{{\cal D}}=\Big\{
{\mib\mu}=(\mu_1,\dots,\mu_N)\in
&L_{w*}^{\bar{s}}(\R^d;(C^d(\Pd))')^N:
\\& \sum\limits_{j=1}^N\sum\limits_{k=1}^d (2\pi i \xi_k)^{\alpha_k}
a_{s j k}
  \mu_j=0, \; s=1,\dots,M \Big\},
 \nonumber
\end{align}where the given equality is understood in the sense of $L_{w*}^{\bar{s}}(\R^d;(C^d(\Pd))')$.

Let us assume that
\begin{equation}
\label{r2}
\begin{split}
&\text{{coefficients of the bilinear form $q$ from \eqref{q-form} }}\\
&\text{{belong to the space $L^t(\R^d)$, where $t\geq \bar{s}'$.}}
\end{split}
\end{equation} Remark that since $\bar{s} \in \langle 1,\frac{pq}{p+q}\rangle$ and $t\geq \bar{s}'$, it also must be $1/t + 1/p + 1/q <
1$.

\begin{definition}
\label{def_sc} We say that the set $\Lambda_{{\cal D}}$, bilinear
form $q$ from \eqref{q-form} satisfying \eqref{r2}, and the matrix
${\mib\mu}=[\mu_{j m}]_{j,m=1,\dots,N}$, $\mu_{j m}\in
L_{w\star}^{\bar{s}}(\R^d;(C^d(\Pd))')$ satisfy {\sl the strong
consistency condition} if for every fixed $m\in \{1,\dots,N \}$, the
N-tuple $(\mu_{1m},\dots,\mu_{N m})$ belongs to $\Lambda_{{\cal
D}}$, and it holds
\begin{equation}
\label{cc11}
\sum\limits_{j,m=1}^N \langle \phi q_{jm}  \otimes 1, \mu_{jm}\rangle \geq 0, \ \ \phi\in \Cbc{\R^d;\R^{+}_{0}}.
\end{equation}

\end{definition} Under the given strong consistency condition, we have the following
theorem.

\begin{theorem}
\label{ccpq} Assume that sequences $(\u_r)$ and $(\vv_r)$ are bounded
in $L^p(\R^d;\R^N)$ and $L^{q}(\R^d;\R^N)$, where $1/p+1/q < 1$, and converge toward $\u$ and $\vv$ in the sense of
distributions. Assume that \eqref{fract} holds.

Assume that
$$
q(\mx;\u_r,\vv_r) \rightharpoonup \omega \ \ {\rm in} \ \ {\cal
D}'(\R^d)
$$ for the bilinear form $q$ from \eqref{q-form} satisfying
\eqref{r2}.

If the set $\Lambda_{{\cal D}}$, the bilinear form \eqref{q-form},
and the (matrix of) $H$-distributions ${\mib\mu}$ corresponding to the
sequences $(\u_r-\u)$ and $(\vv_r-\vv)$ satisfy the strong consistency
condition, then it holds
\begin{equation}
\label{concl1}
q(\mx;\u,\vv)\leq \omega \ \ {\rm in} \ \ {\cal D}'(\R^d).
\end{equation}

If in \eqref{cc11} stands equality, then we have equality in \eqref{concl1} as well.

\end{theorem}

\begin{proof}

%Using \eqref{fract}, we shall first deduce the localisation principle corresponding to the $H$-distributions $\mu_{\alpha \beta}$, $\alpha,\beta=1,\dots, N$.

Let us abuse the notation by denoting $\u_r=\u_r-\u \rightharpoonup
{\mib 0}$ and $\vv_r=\vv_r-\vv \rightharpoonup {\mib 0}$ as $r\to
\infty$.

Remark that, according to Theorem \ref{prop_repr1}, for any non-negative $\phi\in {\cal D}(\R^d)$
\begin{equation}
\label{zg3} \lim\limits_{r\to \infty}\int_{\R^d}
\sum\limits_{j,m=1}^N q_{j m} u_{j r} v_{m r} \phi\ d\mx=\langle
\phi \sum\limits_{j,m=1}^N q_{j m}\otimes 1, \mu_{j m}\rangle,
\end{equation} where $\mu_{jm}$ is a H-distribution corresponding to sequences $u_{j
r}, v_{m r} \rightharpoonup 0$. Since, according to the
localisation principle \eqref{vv2}, for every fixed $m\in
\{1,\dots,N\}$, the $N$-tuple $(\mu_{1m},\dots,\mu_{N m})$ belongs
to $\Lambda_{{\cal D}}$, we conclude from the strong consistency
condition that
$$
\langle \phi \sum\limits_{j,m=1}^N q_{jm} \otimes 1, \mu_{jm}\rangle \geq 0.
$$ From here, \eqref{zg3}, and the fact that (since $q$ is bilinear)

$$
q(\mx;\u_r,\vv_r) \rightharpoonup \omega-q(\mx;\u,\vv)\geq 0 \ \ {\rm in} \ \ {\cal D}'(\R^d),
$$ the statement of the theorem follows.
\end{proof}

If we assume that the sequence $(\vv_n)$ is bounded in
$L^{p'}(\R^d;\R^N)$ and additionally assume that it can be well
approximated by the truncated sequence $(T_l(\vv_n))$, $l\in \N$, we
can state the optimal variant of the compensated compactness as
follows.

\begin{corollary}
\label{optimal}

Assume that

\begin{itemize}
\item sequences $(\u_r)$ and $(\vv_r)$ are bounded in $L^p(\R^d;\R^N)$
and $L^{p'}(\R^d;\R^N)$, where $1/p+1/p' = 1$, and converge toward
$\u$ and $\vv$ in the sense of distributions;

\item for every $l\in \N$, the sequences $(T_l(\vv_r))$ converge weakly in
$L^{p'}(\R^d;\R^N)$ toward ${\bf h}^l$, where the truncation
operator $T_l$ from \eqref{trunc} is understood coordinatewise;

\item there exists a vector valued function ${\bf V}\in L^{p'}(\R^d;\R^N)$
such that $|\vv_r| \leq {\bf V}$ holds coordinatewise for every
$r\in \N$;

\item  \eqref{fract} holds with $a_{skl}\in C_0(\Rd)$ and $q_{jm}\in C(\Rd)$.

\end{itemize}

Assume that
$$
q(\mx;\u_r,\vv_r) \rightharpoonup \omega \ \ {\rm in} \ \ {\cal
D}'(\R^d).
$$

If for every $l\in \N$, the set $\Lambda_{{\cal D}}$, the bilinear
form \eqref{q-form}, and the (matrix of) $H$-distributions ${\mib
\mu}_l$ corresponding to the sequences $(\u_r-\u)$ and
$(T_l(\vv_r)-{\bf h}^l)_r$ satisfy the strong consistency condition,
then it holds
\begin{equation}
\label{concl1opt}
q(\mx;\u,\vv)\leq \omega \ \ {\rm in} \ \ {\cal
D}'(\R^d).
\end{equation}

If in \eqref{cc11} stands equality, then we have equality in
\eqref{concl1} as well.

\end{corollary}

\begin{proof}
For every $l\in \N$, notice that $(q(\mx;\u_r,T_l(\vv_r)))_r$ is
bounded in $L^p(\R^d)$:
\begin{align*}
\int_\Rd |q(\mx;\u_r,T_l(\vv_r))|^p d\mx &\leq N^{2(p-1)} \sum_{j,m = 1}^N \int_\Rd|q_{jm}|^p |u_{jr}|^p |T_l(v_{mr})|^p d\mx\\
&\leq C_{N,l,p} \max_{j,m}(\|q_{jm}\|^p_{{\rm L}^\infty(K)}
\|u_{jr}\|^p_{{\rm L}^p(K)}),
\end{align*} where $K\subseteq\Rd$ is a compact set (remember that sequences
$(\u_r)$, $(\vv_r)$ are uniformly compactly supported). Therefore,
the sequence $(q(\mx;\u_r,T_l(\vv_r)))$ (we remind that $l$ is
fixed) admits a weak limit in $L^p(\R^d)$ (and thus in ${\cal
D}'(\R^d)$) along a subsequence. Using a diagonal procedure, we can
extract a subsequence (not relabeled) such that for every $l\in\N$
it holds
$$
q(\mx;\u_r,T_l(\vv_r)) \rightharpoonup \omega_l \ \ {\rm in} \ \ {\cal D}'(\R^d).
$$ where $\omega_l$ is a weak limit of $(q(\mx;\u_r,T_l(\vv_r)))_r$.
According to the assumptions of the corollary on the strong
consistency conditions involving ${\mib \mu}_l$ and the sequences
$(\u_r-\u)$ and $(T_l(\vv_r)-{\bf h}^l)_r$, and Theorem \ref{ccpq}
(remark that $(T_l(\vv_r))_r$ is bounded), it holds
\begin{equation}
\label{l+} q(\mx;\u,{\bf h}^l)\leq \omega_l \ \ {\rm in} \ \ {\cal
D}'(\R^d).
\end{equation} We will finish the corollary if we show that for every
nonnegative function $\varphi\in C^\infty_c(\Rd)$ it holds $\int_\Rd (\omega
- q(\mx;\u,\vv))\varphi d\mx\geq 0$. It holds

\begin{align}
\label{2014_1}
\int_\Rd (\omega - q(\mx;\u,\vv))\varphi d\mx & = \int_\Rd (\omega - q(\mx;\u_r,\vv_r))\varphi d\mx\\
&+\int_\Rd (q(\mx;\u_r,\vv_r) - q(\mx;\u_r,T_l(\vv_r)))\varphi d\mx \nonumber\\
& + \int_\Rd (q(\mx;\u_r,T_l(\vv_r)) - \omega_l )\varphi d\mx + \int_\Rd (\omega_l - q(\mx; \u,{\bf h}^l))\varphi d\mx\nonumber\\
& + \int_\Rd (q(\mx; \u, {\bf h}^l) - q(\mx;\u,\vv))\varphi d\mx.
\nonumber
\end{align} Since the left hand side of \eqref{2014_1} does not depend on $r$
and $l$, we can take $\limsup\limits_{l\to\infty}\lim\limits_{r\to
\infty}$ there. The first summand on the right hand side of the
expression goes to zero according to the assumptions of the
corollary; the third summand goes to zero according to the
definition of $\omega_l$; we have established in \eqref{l+} that the
fourth summand is nonnegative. Let us show that the second
summand in \eqref{2014_1} goes to zero:
\begin{align*}
\Big| \int_\Rd (q(\mx;\u_r,\vv_r) - q(\mx;\u_r,T_l(\vv_r)))\varphi d\mx\Big| &\leq \int_\Rd |\varphi \,{\bf Q}\u_r\cdot(\vv_r - T_l(\vv_r))| d\mx\\
&\leq  \|{\bf Q}\u_r \|_{{\rm L}^p}\|\varphi\,(\vv_r - T_l(\vv_r))\|_{{\rm L}^{p'}},
\end{align*}where we have used the H\"older inequality. Since $\vv_r - T_l(\vv_r)\to 0$ pointwise, according to the assumption $|\vv_r| \leq {\bf V}$
and the Lebesgue dominated convergence theorem, we conclude that
$\|\varphi\,(\vv_r - T_l(\vv_r))\|_{{\rm L}^{p'}} \to 0$ as $l,r \to
\infty$ (or as $l\to \infty$ uniformly with respect to $r$).

\noindent For the last summand, we will proceed in a similar manner. Let us notice that we can write
\begin{align*}
q(\mx;\u,{\bf h}^l) - q(\mx;\u,\vv) &= {\bf Q}\u\cdot ({\bf h}^l - \vv)\\
&= {\bf Q}\u\cdot \big( ({\bf h}^l - T_l(\vv_r)) + (T_l(\vv_r) - \vv_r) + (\vv_r - \vv)\big).\\
\end{align*}
The first and the last summand on the right hand side of the last
expression will go to zero according to the assumptions of the
corollary. Concerning the second summand, from the Lebesgue dominated convergence theorem as before, we conclude
$\limsup\limits_{l\to\infty}\lim\limits_{r\to \infty}\|(T_l(\vv_r) -
\vv_r)\varphi\|_{L^1(\R^d)}=0$. This concludes the proof.
\end{proof}

\begin{remark}
The condition concerning existence of the dominating function ${\bf V}$ from the previous theorem might look superfluous. However, as the following example shows, we cannot avoid it. Indeed, consider the case $d=N=1$, $a=a_{111}=0$. Let
$$
u_r(\mx)=v_r(\mx)=\begin{cases}
r, & |x|<r^{-2}\\
0, & |x|\geq r^{-2}
\end{cases}.
$$ Then, $\|u_r\|_2=2$ for all $r\in \N$. Clearly, $u_r=v_r\rightharpoonup 0$ weakly as $r\to \infty$, while $T_l(u_r)\to 0$ as $r\to \infty$ strongly in $L^2(\R)$ for every $l\in \N$. Therefore, the $H$-distributions $\mu_l$ corresponding to the sequences $(u_r)$ and $(T_l(v_r))$ are trivial: $\mu_l \equiv 0$. Thus, the strong consistency condition is satisfied with the equality sign, but $q(u_r,v_r)=u_r^2 \rightharpoonup 2\delta(\mx)\neq 0=q(0,0)$.

We would like to thank to the referee for this example.
\end{remark}

In a conclusion of the section, we would like to make a comment
concerning a connection between the standard consistency condition
and, at least at first sight stronger, the strong consistency
condition. To this end, note that we can rewrite the consistency
condition \eqref{(3)} in the following form (we shall omit the
second order derivatives since they have no influence on the
reasoning below):
$$
\Lambda_{{\cal F}}=\Big\{{\mib\lambda}:\R^d\times S^{d-1}\to \R^N: \;
\sum\limits_{j=1}^N \sum\limits_{k=1}^\nu a_{s j k}(\mx)\xi_k\lambda_j(\mx,\mxi)=0, \; s=1,\dots,M \Big\}
$$ and
$$
q(\mx;{\mib\lambda}(\mx,\mxi),{\mib\lambda}(\mx,\mxi))\geq 0 \ \
\text{ for all ${\mib\lambda} \in \Lambda_{{\cal F}}$ and all $(\mx,\mxi) \in \R^d\times S^{d-1}$}.
$$ Having such a representation of the consistency condition, it seems
reasonable to ask whether $\Lambda_{{\cal D}}$ is a closure of
$\Lambda_{{\cal F}}$ in the sense of distributions. If this is the
case, the generalisation presented here holds under the standard
consistency condition. At this moment, we do not have any answer to
this question.

However, we shall present an example showing that our approach can
be used.

\section{Application}

Let us consider the non-linear parabolic type equation

\begin{align}
\label{zg5} L(u)=\pa_t u-\sum\limits_{k,l=1}^d \pa_{x_l x_k}
(a_{kl}(t,\mx)g(t,\mx,u))
\end{align} on $\Omega=\langle0,\infty\rangle\times V$, where $V$ is an open subset of $\R^d$. We assume that
\begin{align*}
u\in L^p(\Omega), \ \ g(t,\mx,u)\in L^q(\Omega), \ \ 1<p,q,\\
a_{kl}\in L^s_{loc}(\Omega), \ \ {\rm where} \ \ 1/p+1/q+1/s<1,
\end{align*} and that the matrix function ${\bf A}=[a_{kl}]_{k,l=1,\dots,d}$ is strictly positive
definite on $\Omega$, i.e.
$$
{\bf A}\mxi \cdot \mxi >0, \ \  \mxi\in \R^d\setminus\{{\mib 0}\},  \ \ \text{a.e.}\ (t,\mx)\in \Omega.
$$
Furthermore, assume that $g$ is a Carath\`{e}odory function and non-decreasing with respect to the third variable.

The following theorem holds.

\begin{theorem}
\label{application} Assume that sequences

\begin{itemize}
\item $(u_r)$ and $g(\cdot,u_r)$ are such that $u_r, g(u_r)\in
L^2(\R^+\times\R^d)$ for every $r\in \N$;

\item that they are bounded in $L^p(\R^+\times\R^d)$, $p\in
\langle 1,2]$, and $L^q(\R^+\times\R^d)$, $q>2$, respectively, where
$1/p+1/q <1$;

\item $u_r \rightharpoonup u$ and, for some, $f\in
W^{-1,-2;p}(\R^+\times \R^d)$, the sequence
$$
L(u_r)=f_r \to f \ \ \text{ strongly in $W^{-1,-2;p}(\R^+\times
\R^d)$.}
$$

\end{itemize}
Under the assumptions given above, it holds
$$
L(u)=f \ \ {\rm in} \ \ {\cal D}'(\R^+\times \R^d).
$$

%In addition, the sequence
%$$
%g(t,\mx,u_r)\to g(t,\mx,u) \ \ \text{strongly in} \ \ L^1_{loc}(\R^+\times \R^d).
%$$
\end{theorem}

\begin{proof}
Let us first define all functions on $\R\times\R^d$ by extending
them with 0 out of $\R^+\times \R^d$. Denote by $w$ a distributional
limit of $g(\cdot,u_r)$ along not relabeled subsequence. Our first
step is to show that the product of $u_r$ and $g(\cdot,u_r)$
converges to $uw$ in the sense of distributions. To do that, denote
\begin{equation}
\label{wl} u_{1r}=u_r-u, \ \ u_{2r}=g(\cdot,u_r)-w.
\end{equation} Note that the following
sequence of equations is satisfied
\begin{equation}
\label{zg6} \pa_t u_{1r}-\sum\limits_{k,l=1}^d \pa_{x_l
x_k}(a_{kl}u_{2r})=f_r-f,
\end{equation} and that  $f_r-f$ tends to zero strongly in $W^{-1,-2;p}(\R^+\times\R^d)$. Introduce

\begin{equation}
\label{skup_D}
\Lambda_{{\cal D}}\!=\!\Big\{{\mib\mu}=(\mu_1,\mu_2)\!\in\!
L_{w\star}^{s'}(\R^+\times\R^d;C^{d+1}(\Pd)')^2\!: \, -2i\pi \xi_0
\mu_1+4\pi^2\sum\limits_{k,l=1}^d \xi_k \xi_l
a_{kl}\mu_{2}\!=\!0\Big\},
\end{equation} and remark that, according to the localisation principle given in
Proposition \ref{localisation},
\begin{equation}\label{mu_iz_lambdaD}(\mu_{12},\mu_{22})\in \Lambda_{{\cal
D}}\end{equation}
for $H$-distributions $\mu_{12}$ and $\mu_{22}$,
corresponding to sequences $(\phi u_{1r})$ and $(\phi u_{2r})$,
and $(\phi u_{2r})$ and $(\phi u_{2r})$, respectively. Above, $\phi\in C^2_c(\R^+\times\R^d)$ is fixed.

From the localisation principle, for $\psi \in C^{d+1}(\Pd)$ (here and
in the sequel, symbols are real functions) and $\varphi\in
C^2_c(\Rd)$, it holds

%(remark that $- \xi_0$ is odd, and $\xi_k \xi_l$, $k,l=1,\dots,d$,
%are even), we conclude that for any odd or even $\psi \in C^d(\Pd)$,
%and any $\varphi \in C^2_c(\R^+\times \R^d)$:

\begin{align}
\label{loc_par} &i \langle -2\pi \xi_0 \psi \varphi, \mu_{12}
\rangle+\langle 4\pi^2\sum\limits_{k,l=1}^d \xi_k \xi_l
a_{kl}(\cdot,\cdot)\psi \varphi, \mu_{22} \rangle=0 .
\end{align} Remark that for any $\psi \in C^{d+1}(\Pd)$ the function
$f_\psi=\langle \psi, \mu_{j2} \rangle$ is in $L^{s'}(\R^+\times
\R^d)$, $j=1,2$. For the functions $f_\psi$, where $\psi$ belongs to a
dense countable subset $E$ of $C^{d+1}(\Pd)$ containing a dense subset of
odd and even functions (which we may choose since $C^{d+1}(\Pd)$ is
separable and we can represent every function as a sum of even and
odd functions
$\psi(\mxi)=\frac{1}{2}(\psi(\mxi)+\psi(-\mxi))+\frac{1}{2}(\psi(\mxi)-\psi(-\mxi))$),
and the functions $a_{kl}$, $k,l=1,\dots,d$, denote by $D\subseteq
\R^+\times \R^d$ the set of their common Lebesgue points (which is of full measure).

Now, fix $(t_0,\mx_0)\in D$. According to the Plancherel theorem, we
get
\begin{equation}\label{pomocno}
\int \overline{\varphi v}{\cal A}_\psi(\varphi v) = \int \overline{\widehat{\varphi v}}\; \psi\widehat{\varphi v}  \in \R
%\overline{\int \overline{\psi}\widehat{\varphi v}\; \widehat{\overline{\varphi v}}} =
%\overline{\int \overline{\varphi v} {\cal A}_\psi(\varphi v)},
\end{equation} for all $v\in L^2(\R^+\times\Rd)$, real bounded multipliers $\psi$,
and $\varphi\in C^2_c(\Rd)$. From here we conclude that
\begin{equation}
\label{real} \langle 4\pi^2\sum\limits_{k,l=1}^d \xi_k \xi_l
a_{kl}(t_0,\mx_0)\psi \varphi, \mu_{22} \rangle \in \R
\end{equation} for any real multiplier $\psi$.
Indeed, for a scalar matrix ${\bf A}(t_0,\mx_0)$, taking into
account that $4\pi^2{\bf A}(t_0,\mx_0)\mxi\cdot\mxi\geq 0$, we
notice that $4\pi^2{\bf
A}(t_0,\mx_0)\mxi\cdot\mxi\psi\boxtimes\varphi$ is a real function
in $\mxi$ (where $\varphi$ is constant with respect to $\mxi$).
Insert symbol $4\pi^2({\bf A}(t_0,\mx_0)\mxi\cdot\mxi\psi /
\rho_{\Pd}) \boxtimes\varphi$ and sequences $u_r = v_r = \phi u_{2r}$
into definition \eqref{rev2} of H-distributions where
$$
\rho_{\Pd}=(\xi_0^2+\sum\limits_{j=1}^d \xi_j^4)^{1/2}.
$$ Now, the
claim follows once we notice that equation \eqref{pomocno} gives us
a limit of real numbers.

On the other hand, from
Lemma \ref{r-c}, we conclude that for any odd $\psi$, the function
\begin{equation}
\label{complex} \langle 4\pi^2\sum\limits_{k,l=1}^d \xi_k \xi_l
a_{kl}(t_0,\mx_0)\psi \varphi, \mu_{22} \rangle \in i \R.
\end{equation} Thus, from \eqref{real} and \eqref{complex}, we
conclude that for any odd function $\psi$ it must be
\begin{equation}
\label{zakljucak}\langle 4\pi^2\sum\limits_{k,l=1}^d \xi_k \xi_l
a_{kl}(t_0,\mx_0)\psi \varphi, \mu_{22} \rangle=0.
\end{equation}

%\noindent Adding \eqref{zakljucak} to both sides of \eqref{loc_par}, we can rewrite it in the following form
%
%\begin{align}
%\label{26*} &i \langle -2\pi \xi_0 \psi \varphi, \mu_{12}
%\rangle+\langle 4\pi^2\sum\limits_{k,l=1}^d \xi_k \xi_l
%a_{kl}(t_0,\mx_0)\psi \varphi, \mu_{22} \rangle\\
%\notag &=\langle 4\pi^2\sum\limits_{k,l=1}^d \xi_k \xi_l
%(a_{kl}(t_0,\mx_0)-a_{kl}(\cdot,\cdot))\psi \varphi, \mu_{22}
%\rangle.
%\end{align}

\noindent Taking into account \eqref{zakljucak}, assuming $\psi\in E$, and inserting
$(t,\mx)=(t_0,\mx_0)$ into \eqref{loc_par}, we conclude that for all points from $D$, it
holds
\begin{equation}
\label{xi0} \langle -2\pi \xi_0 \psi, \mu_{12}(t_0,\mx_0,\cdot)
\rangle=0.
\end{equation} Now, since $u_r\in L^2(\R^+\times \R^d)$ for every
$r\in \N$, we can test \eqref{zg6} by $\varphi \overline{{\cal
A}_{(1-\theta)\psi_\Pd/\rho_{\Pd}}(\varphi u_{1r})}$ where $\theta$
is a compactly supported even smooth function equal to one in a
neighborhood of zero. Then, we let $r\to \infty$ and use the
Plancherel theorem to obtain a relation similar to \eqref{loc_par}
(remark that ${{\cal A}_{(1-\theta)\psi_\Pd/\rho_{\Pd}}}$ is a
compact $L^p\to L^p$ operator for any $p>1$):

\begin{align}
\label{pomocna2} &\lim\limits_{r\to \infty}\int_{\R^{d+1}}-2\pi
i\frac{(1-\theta(\mxi))\xi_0}{\rho_{\Pd}(\mxi)}
\psi_{\Pd}(\mxi)\F({\varphi u_{1r}})\overline {\F({\varphi
u_{1r}})}d\mxi\\&\qquad\qquad\qquad\qquad+\langle
4\pi^2\sum\limits_{k,l=1}^d \xi_k \xi_l a_{kl}(\cdot,\cdot)\psi
\varphi, \mu_{12} \rangle=0, \nonumber
\end{align}where, as usual, $\psi_P=\psi\circ \rho_{\Pd}$. Denote by
\begin{align}
\label{pom5} I_r(\psi_\Pd)\!&=\!\int_{\R^{d+1}} -2\pi
i\frac{(1-\theta(\mxi))\xi_0}{\rho_{\Pd}(\mxi)}
\psi_{\Pd}(\mxi)\F({\varphi u_{1r}})\overline {\F({\varphi
u_{1r}})}d\mxi\!\\ &=\!\int_{\R^{d+1}}-2\pi
i\frac{(1-\theta(\mxi))\xi_0}{\rho_{\Pd}(\mxi)}
\psi_{\Pd}(\mxi)|\F({\varphi u_{1r}})|^2d\mxi. \nonumber
\end{align}We
shall prove that for every even $\psi$
\begin{equation}
\label{pom3}  I_r(\psi_\Pd) = 0.
\end{equation} Clearly, for any
real $\psi$, it holds (see \eqref{pom5})
\begin{equation}
\label{pom4} I_r(\psi_{\Pd}) \in i \R.
\end{equation} However, from Lemma \ref{r-c}, we conclude that for
any even $\psi$, it holds
\begin{align*}
I_r(\psi_{\Pd})&=\int_{\R^+\times \R^{d}}
\varphi(\mx)u_{1r}(t,\mx) \pa_t \overline{\big({\cal A}_{(1-\theta)\psi_\Pd/\rho_{\Pd}}
(\varphi u_{1r})\big)(t,\mx)} dt d\mx\\&=
\int_{\R^+\times\R^{d}} \varphi(\mx)u_{1r}(t,\mx) \pa_t
{\big({\cal A}_{(1-\theta)\psi_\Pd/\rho_{\Pd}} (\varphi u_{1r})\big)(t,\mx)} dt d\mx \in
\R.
\end{align*}Being both purely real for any even $\psi$ and purely imaginary for any $\psi$ (see
\eqref{pom4}), it follows that $I_r(\psi_\Pd)$ must be zero for any
even $\psi$. From here, \eqref{pom3} follows.

Now, since the function $\varphi\in C_c^2(\R^{d+1})$ is arbitrary, from \eqref{pomocna2} we get the following relation for every (Lebesgue) point $(t,\mx)\in D$ and $\psi \in E$:

\begin{equation}
\label{xik} \langle 4\pi^2\sum\limits_{k,l=1}^d \xi_k \xi_l
a_{kl}(t,\mx)\psi_2, \mu_{12}(t,\mx,\cdot) \rangle =0.
\end{equation}

Since the set $D$ is of full measure, summing the results from \eqref{xi0} and \eqref{xik}, we conclude that for any odd symbol
$\psi_1 \in E$ and even symbol $\psi_2\in E$, we have

\begin{align*}
&\langle 2\pi \xi_0 \psi_1 \varphi, \mu_{12} \rangle+ \langle
4\pi^2\sum\limits_{k,l=1}^d \xi_k \xi_l a_{kl}(t,\mx)\psi_2 \varphi,
\mu_{12} \rangle =0.
\end{align*}
Thus, by taking $\psi_1=\xi_0 \psi$ and $\psi_2=\psi$ for an even
symbol $\psi \in E$, we conclude:

\begin{equation}
\label{even} \Big\langle \Big( 2 \pi
\xi_0^2+4\pi^2\sum\limits_{k,l=1}^d \xi_k \xi_l
a_{kl}(t,\mx)\Big)\psi \varphi, \mu_{12} \Big\rangle=0.
\end{equation} Since $\mu_{12}$ is continuous on $L^s(\R^{d+1};C^{d+1}(\Pd))$, we conclude that \eqref{even} holds for any even $\psi\in C^{d+1}(\Pd)$.

Since the function
$$
f(t,\mx,\mxi)=\frac{\varphi}{2 \pi
\xi_0^2+4\pi^2\sum\limits_{k,l=1}^d \xi_k \xi_l a_{kl}} \in
L^s(\R^+\times\R^d;C^{d+1}(\Pd))
$$

\noindent is even with respect to the variable $\mxi$, we conclude
from \eqref{even} (we can put $f$ instead $\varphi \psi$ there) that
\begin{equation}
\label{local} \langle 1 \otimes \varphi, \mu_{12} \rangle=0.
\end{equation}

\noindent From
\eqref{mu_iz_lambdaD} and \eqref{local}, we conclude that the following bilinear form
$$
q(\mx;{\mib\lambda},{\mib\eta})=\lambda_{1}\eta_2, \ \
{\mib\lambda}=(\lambda_1,\lambda_2), \; {\mib\eta}=(\eta_1,\eta_2),
$$ satisfies the strong consistency condition with the set $\Lambda_{{\cal
D}}$ introduced in \eqref{skup_D}. Now we can apply Theorem \ref{ccpq} to conclude that
\begin{equation}
\label{bd11} q(\mx;(u_{1r},u_{2r}),(u_{2r},u_{2r})) = u_{1r} u_{2r} \rightharpoonup  0 = q(\mx;(0,0),(0,0)) \ \ {\rm in}
\ \ {\cal D}'(\R^+\times\R^d)
\end{equation} since both $u_{1r} = u_r - u$ and $u_{2r} = g(\cdot,u_r)-w$
weakly converge to $0$. Using the bilinearity of $q$, we conclude

\begin{equation}
\label{bd11f} u_{r} g(\cdot,u_r) \rightharpoonup u w \ \ {\rm in} \
\ {\cal D}'(\R^+\times\R^d).
\end{equation}

Our next step is to identify $g(\cdot,u)$ as a weak limit of
$g(\cdot,u_r)$. To do that we will employ the theory of Young
measures. Up to this moment we didn't need any assumption on the
function $g$ itself, only on the sequence $g(\cdot,u_r)$.

\noindent Denote by $\eta_{t,\mx}$ the Young measure associated to a
subsequence of the sequence $(u_r)$. Since $g$ is a Carath\`{e}odory
function, from \eqref{wl} and \eqref{bd11f}, it holds \cite{PP}:

\begin{equation}
\label{bd12}
\begin{cases}
u(t,\mx)=\int \lambda d \eta_{t,\mx}(\lambda), \\
w(t,\mx)=\int g(t,\mx,\lambda) d\eta_{t,\mx}(\lambda),%, \ \ 0=\int
%(w(t,\mx)-g(t,x,\lambda))(u(t,\mx)-\lambda)d\eta_{t,\mx}(\lambda)
\end{cases}
\end{equation}
and
$$u(t,\mx) \int g(t,\mx,\lambda) d\eta_{t,\mx}(\lambda)=u(t,\mx) w(t,\mx)=\int
\lambda g(t,\mx,\lambda)d \eta_{t,\mx}(\lambda).$$

The latter equality implies
\begin{align}
\label{bd13} \int(\lambda-u(t,\mx))g(t,\mx,\lambda) &d\eta_{t,\mx}(\lambda)=\\
\notag &\int\Big(\lambda-u(t,\mx)\Big)
\Big(g(t,\mx,\lambda)-g(t,\mx,u(t,\mx))\Big)d\eta_{t,\mx}(\lambda)=0,
\end{align}
because

\begin{align}
\notag \int (\lambda - u) g(t,\mx,u)d\eta_{t,\mx}(\lambda)&=
g(t,\mx,u)\int \lambda d\eta_{t,\mx}(\lambda) - g(t,\mx,u)u\int d\eta_{t,\mx}(\lambda)\\
\notag &= g(t,\mx,u)u - g(t,\mx,u)u\\
\notag &= 0,
\end{align}
where function $u$ does not depend on $\lambda$ and we have used first equality in \eqref{bd12} and the fact that $\eta_{t,\mx}$ is a probability measure.

Since $g$ is non-decreasing with respect to
$\lambda$, we conclude from \eqref{bd13}
$$
g(t,\mx,\lambda)=g(t,\mx,u(t,\mx)) \ \ {\rm on} \ \ {\rm supp}
\eta_{t,\mx},$$
which implies
$$w(t,\mx)=\int g(t,\mx,\lambda)d\eta_{t,\mx}(t,\mx)=g(t,\mx,u(t,\mx)).
$$ From here, we finally conclude that
$$
L(u_r) \rightharpoonup L(u)=f \ \ {\rm in} \ \ {\cal D}'(\R^+\times\R^d).
$$ \end{proof}

%Next, fix an $\eps>0$ and denote by $K_\delta \subset \Omega_2$ an
%open set such that ${\rm meas}(\Omega_2\backslash K_{\delta})<\eps$,
%and remark that, ...

%\begin{remark}Remark that Theorem \ref{application} holds if we assume that there exists a partition
%$$
%\Omega=\Omega_1\cup \Omega_2 \cup \Gamma,
%$$ where $\Omega_1$ and $\Omega_2$ are disjoint open sets, while $\Gamma$ is a set of measure zero, such that
%\begin{itemize}
%
%\item the function $g=g(t,\mx,u)$ is linear with respect to $u$ on
%$\Omega_1$, and it is the Carath\`{e}odory function on $\Omega_2\times
%\R$ non-strictly increasing with respect to $u$;
%
%\item for every $\eps>0$ there exists an open set $K_\delta \subset \Omega_2$ such
%that ${\rm meas}(\Omega_2\backslash K_\delta) <\eps$, and the matrix
%$A=[a_{kl}]_{k,l=1,\dots,d}$ is strictly positive definite on
%$K_{\delta}$, i.e. $A\mxi \cdot \mxi >0$, $\mxi\in \R^d\setminus\{{\mib 0}\}$.
%
%\end{itemize}
%Indeed, in such a case,  $L(u_r)\to L(u)$ in ${\cal D}'(\Omega_2)$
%as $r\to \infty$ since $L$ is linear on $\Omega_2$. According to
%Theorem \ref{application}, we conclude that $L(u_r)\to L(u)$ in
%${\cal D}'(K_{\delta})$ as $r\to \infty$ for every $\delta>0$. This
%actually means that $L(u_r)\to L(u)$ in ${\cal D}'(\Omega\backslash
%\Gamma)$ as $r\to \infty$, and since $\Gamma$ is a closed set of
%measure zero, it also holds $L(u_r)\to L(u)$ in ${\cal D}'(\Omega)$
%as $r\to \infty$.
%
%\end{remark}

\section{Acknowledgements}

The authors would like to thank to prof. Nenad Antoni\'{c} and to the unknown referee for their
generous help during the work on the article.

The research is supported by the bilateral project Transport in
highly heterogeneous media between Croatia and Montenegro, by the
Croatian Science Foundation's funding of the project "Weak
convergence methods and applications" number 9780, and by the
project "Advection-diffusion equations in highly heterogeneous
media" of the Montenegrian Ministry of Science. Part of the work was
carried out while M.Mi\v sur was visiting the Basque Center for
Applied Mathematics, within the frame of FP7--246775 NUMERIWAVES
project of the ERC.

\end{document}